\documentclass[12pt]{amsart}
\usepackage{amsfonts,amssymb,latexsym,amsmath}
\usepackage{amsthm}
\usepackage{graphicx}
\usepackage{booktabs}
\usepackage{ytableau}
\usepackage{tikz}
\usepackage{fullpage}

\newtheorem{theorem}{Theorem}
\newtheorem{lemma}[theorem]{Lemma}
\newtheorem{corollary}[theorem]{Corollary}

\theoremstyle{definition}
\newtheorem{definition}[theorem]{Definition}
\newtheorem{example}[theorem]{Example}

\theoremstyle{remark}

\DeclareMathOperator{\rank}{rank}
\DeclareMathOperator{\rk}{rk}

\begin{document}
\title[Resilience of ranks of higher inclusion matrices]{Resilience of ranks of higher inclusion matrices}

\author[Plaza and Xiang]{Rafael Plaza, Qing Xiang$^{\dagger}$}

\thanks{$^{\dagger}$Research partially supported by an NSF grant DMS-1600850}

\address{Rafael Plaza, Department of Mathematical Sciences, University of Delaware, Newark, DE 19716, USA}
\email{plaza@udel.edu}

\address{Qing Xiang, Department of Mathematical Sciences, University of Delaware, Newark, DE 19716, USA} \email{qxiang@udel.edu}


\begin{abstract}
Let $n \geq r \geq s \geq 0$ be integers and $\mathcal{F}$ a family of $r$-subsets of $[n]$. Let $W_{r,s}^{\mathcal{F}}$ be the higher inclusion matrix of the subsets in ${\mathcal F}$ vs. the $s$-subsets of $[n]$. When $\mathcal{F}$ consists of all $r$-subsets of $[n]$, we shall simply write $W_{r,s}$ in place of $W_{r,s}^{\mathcal{F}}$. In this paper we prove that the rank of the higher inclusion matrix $W_{r,s}$ over an arbitrary field $K$ is resilient. That is,  if the size of $\mathcal{F}$ is ``close''  to ${n \choose r}$ then $\rank_{K}( W_{r,s}^{\mathcal{F}}) = \rank_{K}(W_{r,s})$, where $K$ is an arbitrary field. Furthermore, we prove that the rank (over a field $K$) of the higher inclusion matrix of $r$-subspaces vs. $s$-subspaces of an $n$-dimensional vector space over $\mathbb{F}_q$ is also resilient if ${\rm char}(K)$ is coprime to $q$.

\end{abstract}

\maketitle

\section{Introduction}

Let $n\geq r \geq s\geq 0$ be integers, and let $[n]=\{1,2,\ldots ,n\}$. Given a family $\mathcal{F}$ of $r$-subsets of $[n]$, we define the higher inclusion matrix $W_{r,s}^{\mathcal{F}}$ to be the $(0,1)$-matrix with rows indexed by the $r$-subsets $R$ in $\mathcal{F}$, columns indexed by the $s$-subsets $S$ of $[n]$,  and with $(R,S)$-entry equal to one if and only if $S\subseteq R$.  When $\mathcal{F}= {[n] \choose r}$, that is, $\mathcal{F}$ consists of all $r$-subsets of $[n]$, we shall omit the superscript and simply write $W_{r,s}$ in place of $W_{r,s}^{\mathcal{F}}$.

The higher inclusion matrices $W_{r,s}^{\mathcal{F}}$ have played an important role in the theory of $t$-designs (\cite{gj}, \cite{n2}) and in extremal combinatorics (\cite{babaif}, \cite{godsil}). For applications to integral $t$-designs, Wilson \cite{n2} found a diagonal form of $W_{r,s}$. As a consequence, he obtained the rank of $W_{r,s}$ over any field $K$. Specifically, if $n\geq r+s$, and $K$ is any field, then
\[
\rank_{K}(W_{r,s}) = \sum_{j \in Y }\left( {n \choose j} - {n \choose j-1}\right),
\]
where $Y=\{ j :   0 \leq j \leq s,  {r- j \choose s-j }  \neq_K 0   \}$. In the above rank formula, ${n \choose -1}$ should be interpreted as zero. We remark that the above result on the rank of $W_{r,s}$ includes the result of Gottlieb \cite{n4} and the result of Linial and Rothchild \cite{n5} as special cases.

Higher inclusion matrices have also proven very useful in applications of linear algebraic methods in extremal combinatorics (see \cite{babaif}). For example, the following classical result in extremal combinatorics, known as the Lov\'{a}sz version of the Kruskal-Katona theorem \cite{Katona, Kruskal}, can be proved using properties of higher inclusion matrices. Let $\mathcal{F}$ be a family of $r$-subsets of $[n]$. The {\it $s$-shadow of $\mathcal{F}$}, denoted by $\partial^r_s \mathcal{F}$, consists of all $s$-subsets of $[n]$ that are contained in some element of $\mathcal{F}$. 

\begin{theorem}\label{int_teo1}
{\em (Lov\'{a}sz  \cite{Lo2})} Let $\mathcal{F}$ be a family of $r$-subsets of $[n]$ such that $|\mathcal{F}| = {x \choose r}$, where $x$ is a real number greater than or equal to $r$. If $s < r$ then $|\partial^r_s \mathcal{F}| \geq {x \choose s}$, and equality holds if and only if $x$ is an integer and there exists a subset $X$ of $[n]$ of size $x$ such that $\mathcal{F}={X \choose r}$. 
\end{theorem}

The above theorem can be proved in several different ways. Keevash \cite{n6} showed that Theorem 1 follows immediately from the following result on the rank of higher inclusion matrices.

\begin{theorem}\label{int_teo2}
{\em (Keevash \cite{n6})} For every $r > s > 0$ there is a number $n_{r,s}$ so that if $\mathcal{F}$ is a family of $r$-subsets of $[n]$ with $|\mathcal{F}| = {x \choose r} \geq n_{r,s}$ then $\rank_{\mathbb{Q}}(W_{r,s}^{\mathcal{F}}) \geq {x \choose s}$, and equality holds if and only if $x$ is an integer and there exists a subset $X$ of $[n]$ of size $x$ such that $\mathcal{F}={X \choose r}$.
\end{theorem}

To see how Theorem \ref{int_teo1} follows from Theorem \ref{int_teo2}  (for large $x$), one simply observes that ${\rm rank}_{\mathbb Q}(W^{\mathcal{F}}_{r,s})$ is less than or equal to the number of nonzero columns of $W_{r,s}^{\mathcal{F}}$ (which is the size of the $s$-shadow of ${\mathcal F}$). In order to prove Theorem \ref{int_teo2}, Keevash \cite{n6} showed that the rank of the matrix $W_{r,s}$ is {\it resilient} or {\it robust}, that is, one can remove ``many" rows (in an arbitrary way) of $W_{r,s}$ without lowering its rank.

\begin{theorem}\label{r_teo1}
{\em (Keevash \cite{n6})} Suppose $0 \leq s \leq r$ and $2r+s \leq n$. If $\mathcal{F}$ is a family of $r$-subsets of $[n]$ with $|{[n] \choose r} \setminus \mathcal{F}| \leq {n \choose s}^{-1} {n \choose r-s}$ then $\rank_{\mathbb{Q}}(W_{r,s}^{\mathcal{F}})= {n \choose s}$. 
\end{theorem} 

Keevash \cite{Keevash} went further to ask whether Theorem \ref{r_teo1} remains true under the assumption that $|{[n] \choose r} \setminus \mathcal{F}| < {n-s \choose r-s}$. This question was answered in the affirmative by Grosu, Person and Szab\'o \cite{n7} for $n$ large (compared with $r$ and $s$).  In the end of \cite{n7}, the authors remarked that rank resilience property of the higher inclusion matrices has not been studied over fields of positive characteristic. In this paper we prove that the rank of $W_{r,s}$ is resilient over any field $K$.  In fact, the following theorem shows that if the size of $\mathcal{F}$ is close to ${n \choose r}$ then $ \mbox{rank}_{K}( W_{r,s}^{\mathcal{F}})=\mbox{rank}_{K}(W_{r,s})$ for an arbitrary field $K$. To simplify notation,  for any family $\mathcal{F}$ of $r$-subsets of $[n]$, we use $\mathcal{F}^c$ to denote the complement of $\mathcal{F}$ in ${[n] \choose r}$. Our first main result is stated below.

\begin{theorem}\label{r_teo2}  
Assume that $0 \leq s <  r \leq n/2$. Let $\mathcal{F}$ be a family of $r$-subsets of $[n]$, and $K$ be any field. If $|\mathcal{F}^c|\leq \frac{n-1}{r}$ then $ \rank_{K}( W_{r,s}^{\mathcal{F}})=\rank_{K}(W_{r,s})$.
\end{theorem}

Our second main result in this paper is about rank resilience property of the higher inclusion matrices of $r$-subspaces vs. $s$-subspaces of an $n$-dimensional vector space over $\mathbb{F}_q$. 

\begin{definition}
Let $V$ be an $n$-dimensional vector space over $\mathbb{F}_q$, where $q=p^t$ is a prime power. Let $n \geq r \geq s \geq 0$ be integers and $\mathcal{F}$ a family of $r$-dimensional subspaces of $V$. The higher inclusion matrix of $r$-subspaces vs. $s$-subspaces, denoted by $W_{r,s}^{\mathcal{F}}(q)$, is the $(0,1)$-matrix with rows indexed by the $r$-dimensional subspaces $R$ of $V$, columns indexed by the $s$-dimensional subspaces $S$ of $V$, and with the $(R,S)$-entry equal to one if and only if $S\subseteq R$. In the case when $\mathcal{F}= {V \brack r}$, that is, $\mathcal{F}$ consists of all $r$-subspaces of $V$, we shall omit the superscript and simply write $W_{r,s}(q)$.
\end{definition}

The ranks of the matrices $W_{r,s}(q)$ have also been studied. However, the results are not as complete as in the set case. It was proven by Kantor \cite{n8} that if $s \leq \min\{r, n-r\}$ then the $\mathbb{Q}$-rank of $W{_{r,s}}(q)$ is ${n \brack s}$ (the number of $s$-dimensional subspaces in $V$).  Later, Frumkin and Yakir \cite{n10} proved that if $char(K)\neq p$, and $n \geq r+s$ then the $K$-rank of $W_{r,s}(q)$ is given by a $q$-analogue of Wilson's formula. Indeed,
\begin{equation}\label{q_rank}
\rank_{K}(W_{r,s}(q)) = \sum_{j\in Y }\left({n \brack i} - {n \brack i-1}\right),
\end{equation}
where $Y=\{i: 0 \leq i \leq s, {r-i \brack s-i} \neq_K 0 \}$. When the characteristic of $K$ is equal to $p$, the problem of finding the $K$-rank of $W_{r,s}(q)$ is open in general.  However, under the additional condition that $s=1$, Hamada \cite{Hamada} gave a formula for the $p$-rank of $W_{r,1}(q)$.

It is important to remark that although there are at least five different proofs (\cite{n1,n3,n20,n10,n2}) of Wilson's rank formula, only the proof by Frumkin and Yakir \cite{n10} has been generalized to find a formula for the rank of the matrix $W_{r,s}(q)$ over $K$ when $char(K)\neq p$. This is an indication that proving $q$-analogues of classical results in extremal set theory is often a difficult task.

In this paper, we prove that the $K$-rank of  $W_{r,s}(q)$ is also resilient when $char(K) \neq p$. Let $\mathcal{F}$ be a family of $r$-subspaces of $V$. We denote by $\mathcal{F}^c$  the complement of $\mathcal{F}$ in ${V \brack r}$. 

\begin{theorem}\label{r_teo3}
Let $V$ be an $n$-dimensional vector space over $\mathbb{F}_q$. Assume that $0 \leq s < r \leq n/2$. Let $\mathcal{F}$ be a family of $r$-subspaces of $V$ and $K$ a field with $\mbox{char}(K) \neq p$. If $|\mathcal{F}^c| \leq \frac{n}{r}-1$ then $ \rank_{K}( W_{r,s}^{\mathcal{F}}(q))=\rank_{K}(W_{r,s}(q))$.
\end{theorem} 

The techniques we use to prove Theorem~\ref{r_teo2} and Theorem~\ref{r_teo3} are completely different from those used by Keevash in \cite{n6} and Grosu, Person and Szab\'{o} in \cite{n7}. The main tool we use to prove Theorem \ref{r_teo2} is Bier's bases which give a diagonal form of the higher inclusion matrix $W_{r,s}$. These bases were found by Bier in \cite{n1}. We will show that if the size of $\mathcal{F}$ is close to ${n \choose r}$ then Bier's bases also give an almost diagonal form for the matrix $W_{r,s}^{\mathcal{F}}$. This fact will be used to compute the rank of $W_{r,s}^{\mathcal{F}}$. 

The proof of Theorem \ref{r_teo3} is more difficult. One difficulty is that there is no known $q$-analogue of the Bier basis for us to use. To overcome this difficulty we use some results from representation theory of $GL(n,q)$. The work of James \cite{n11} and Frumkin and Yakir \cite{n10} explicitly shows a connection between the rank of higher inclusion matrices and the Specht modules of $GL(n,q)$.  In fact, Frumkin and Yakir \cite{n10} proposed a uniform approach to finding ranks of both $W_{r,s}$ and $W_{r,s}(q)$. The basic idea is that $W_{r,s}$ and $W_{r,s}(q)$ are matrices associated with an $S_n$- and  a $GL(n,q)$-module homomorphisms, respectively. From this point of view, one can use some properties of the Specht modules of $GL(n,q)$ to prove that the column space of  $W_{r,s}^{\mathcal{F}}(q)$ contains at least $\rank_K(W_{r,s}(q))$ linearly independent vectors if the size of $\mathcal{F}^c$ is small enough. To be specific, the properties of the Specht modules that we use are the Submodule Theorem (see Theorem \ref{M_submodule_teo1}) and the {\it standard  bases} for the $GL(n,q)$-Specht modules $S^{(n-r,r)}$, with $r \leq n/2$, that were found by Brandt, Dipper, James and Lyle in \cite{n9}.   Once we prove the result on the column space of $W_{r,s}^{\mathcal{F}}(q)$, Theorem \ref{r_teo3} follows easily since the rank of $W_{r,s}^{\mathcal{F}}(q)$ is clearly bounded above by the rank of $W_{r,s}(q)$. 

\section{Rank Resilience: the Set Case}

\subsection{Bier's Bases}

Let $K$ be an arbitrary field. For any $0 \leq r \leq n$, we denote by  $M^r$ the $K$-vector space spanned by the $r$-subsets of $[n]$. Hence, the set of $r$-subsets of $[n]$ forms a ``canonical" basis of $M^r$. Let $\varphi_{j,r} : M^j \rightarrow M^r$ be the linear transformation such that, for every $j$-subset $A$ of $[n]$, $$\varphi_{j,r}(A) = \sum_{A \subseteq R} R,$$ where the sum is over all $r$-subsets containing $A$; the definition of $\varphi_{j,r}$ is then extended to all elements of $M^j$ by linearity. Note that $W_{r,j}$ is the matrix of $\varphi_{j,r}$ with respect to the canonical bases of $M^j$ and $M^r$.

For any $j$-subset $A$ of $[n]$, with $0 \leq j \leq r$,  we denote by $\langle A \rangle_r$ the image of $A$ under the linear map $\varphi_{j,r}$. In \cite{n3}, Frankl defined the rank of a subset of $[n]$.

\begin{definition}\label{r_definition}
(Frankl \cite{n3}) Let $A$ be a subset of $[n]$. One associates a walk $w(A)$ on the $x$-$y$ plane with $A$. The walk $w(A)$ goes from the origin to $(n-|A|, |A|)$ by steps of length one, with the $i$-th step going east or north according as $i \notin A$ or $i \in A$. The {\it rank} of $A$, denoted by $\rk(A)$, is defined as $|A| -\ell$ where $\ell$ is the largest integer such that $w(A)$ reaches the line $y=x+\ell$.
\end{definition}

From the above definition, it follows that if $A$ is a $j$-subset of $[n]$ then its rank is at most $\min\{j, n-j\}$.  For every $0 \leq j \leq n/2$, we define
\begin{equation*}
    S(j) = \left\{ A \in {[n] \choose j} : \rk(A)=j \right\}.
\end{equation*}
Note that the elements of $S(j)$ are in one-to-one correspondence to the standard tableaux of shape $(n-j,j)$.  This is one way to see that $|S(j)|= {n \choose j} - {n \choose j-1}$.  Therefore, for $0\leq r \leq n/2$, we have $|\cup_{j=0}^r S(j)|= {n \choose r}$, which is precisely the dimension of the vector space $M^r$. The following theorem gives a basis of $M^r$ indexed by the elements of $S(j)$ with $j$ ranging from $0$ to $r$.

\begin{theorem}\label{r_teo4}
{\em (Bier \cite{n1})} Let $0 \leq r \leq n/2$. The vectors in $\cup_{j=0}^r \{ \langle A \rangle_r : A \in S(j)\}$ form a $K$-basis of $M^r$.
\end{theorem}

We will refer to the basis given in Theorem \ref{r_teo4} as the Bier basis of $M^r$. For the sake of completeness we give the details of Bier's proof of Theorem \ref{r_teo4}. 

\begin{lemma}\label{r_lemma1}
{\em (Bier \cite{n1})} Let $r$ be a positive integer. For any $j$-subset $A$ of  $[n]$ with $j < r$,
\begin{equation}\label{r_ecu1}
    {r-j \choose \ell} \langle A \rangle_r + \sum_{i=1}^{\ell} (-1)^i {r-j-i \choose \ell-i} \sum_{T_i} \langle T_i\rangle_r = 0 \quad \mbox{for all } \ell=1, \ldots, r-j
\end{equation}
where the inner sum is over all $T_i$ with $|T_i|=j+i$ and $A \subset T_i$.
\end{lemma}

\begin{proof}
Let $R$ be any $r$-subset containing $A$. In the first term on the left hand side of (\ref{r_ecu1}), $R$ appears ${r-j \choose \ell}$ times. Moreover, in each sum $\sum \langle T_i\rangle_r$, $R$ appears ${r-j \choose i}$ times. Therefore, $R$ appears in the left hand side of (\ref{r_ecu1}) exactly
\[
        {r-j \choose \ell} + \sum_{i=1}^{\ell} (-1)^i {r-j-i \choose \ell-i} {r-j \choose i} = \sum_{i=0}^{\ell} (-1)^i { r-j \choose i} {r-j-i \choose \ell-i}
\]
times. The above sum is easily seen to be zero by the principle of inclusion and exclusion.
\end{proof}

\begin{proof}[\bf Proof of Theorem \ref{r_teo4}]  

We will show that for any $0 \leq t \leq r$
\begin{equation}\label{r_ecu2}
    \mbox{span}_K \{ \langle A \rangle_r : A \in S(j), 0 \leq j \leq t \} = \mbox{span}_K\{ \langle A \rangle_r : A \mbox{ a }j\mbox{-subset of }[n], 0 \leq j \leq t\}.
\end{equation}
The conclusion of the theorem follows immediately from (\ref{r_ecu2}) because by taking $t=r$ we see that the vectors in the set on the left hand side of (\ref{r_ecu2}) span $M^r$ and since $|\cup_{j=0}^r S(j)|= {n \choose r}$, they form a basis.

We will prove (\ref{r_ecu2}) by induction. Let us start with some definitions that we will use. For any set $A=\{a_1 < \cdots < a_j\}$ with $r(A) < |A|$, there exists a unique integer $m=m_A$, $1 \leq m \leq j$ such that $a_m < 2m$ and $a_i \geq 2i$ for all $i > m$. On the other hand, if $r(A)= |A|$ then $a_i \geq 2i$ for all $i$; so $m=m_A=0$ in this case.

To prove (\ref{r_ecu2}) it is enough to show that
\begin{equation}\label{r_ecu3}
    \mbox{span}_K\{ \langle A \rangle_r : A \mbox{ a }j\mbox{-subset of }[n], 0 \leq j \leq s\} \leq \mbox{span}_K\{ \langle A \rangle_r : A \in S(j), 0 \leq j \leq t \}
\end{equation}
for all $s$ ranging from $0$ to $t$. 

The proof of (\ref{r_ecu3}) is done by induction on $s$ and on the parameter $m$ defined above for any subset of $[n]$. Note that the base case, i.e., the case where $s=0$, is trivially true. Now, let $s$ be given with  $0 < s \leq t$ and suppose by induction hypothesis that the following holds:
        \begin{enumerate}
          \item[] (a). $\langle B \rangle_r \in \mbox{span}_K\{\langle A \rangle_r : A \in S(j), 0 \leq j \leq t \}, \mbox{ for all } B, |B| < s$.
          \item[] (b). $\langle B \rangle_r \in \mbox{span}_K\{\langle A \rangle_r : A \in S(j), 0 \leq j \leq t \}, \mbox{ for all } B, |B|=s \mbox{ and } m_B < m$ \footnote{We may assume (b) because for every $s$-subset  $B$ with $m_B=0$ we have that $B \in S(j)$; therefore, $\langle B \rangle_r \in \mbox{span}_K\{\langle A \rangle_r : A \in S(j), 0 \leq j \leq t \}$.}.
        \end{enumerate}
 Using these assumptions we will show that 
 \begin{equation}\label{extra}
        \langle B \rangle_r \in \mbox{span}_K\{\langle A \rangle_r : A \in S(j), 0 \leq j \leq t \}
  \end{equation} 
 for any subset $B$ with $|B|=s$ and $m_B=m$, which is enough to prove (\ref{r_ecu3}).

Let $B=I \cup X$ with
\[
        I=\{ b_1 < b_2 < \cdots < b_m\} \mbox{ and } X=\{ b_{m+1} < \cdots < b_s\}
\]
such that $b_m < 2m$ and $b_i \geq 2i$ for all $b_i \in X$ (so $|B|=s$ and $m_B=m$). For any $U \subseteq I$ we define
\[
    [U \cup X] = \sum_{U \subseteq J} \langle J \cup X \rangle_r
\]
where the sum is taken over all sets $J=\{j_1 < j_2 < \cdots < j_m\}$ with $j_m < 2m$ containing the set $U$. Notice that $J \cup X$ is an $s$-subset with $m_{J \cup X}=m$.

\vspace{0.1in}

\noindent{\bf Claim (i).} Let $U$ be a proper subset of $I$ then
\[
        [U \cup X] \in \mbox{span}_K\{\langle A \rangle_r : A \in S(j), 0 \leq j \leq t \}
\]

To prove the claim, applying Lemma \ref{r_lemma1} with $A=U \cup X$ and $\ell= m - |U|$, we obtain

$${k-|U \cup X| \choose \ell} \langle U \cup X \rangle_r + \sum_{i=1}^{\ell} (-1)^i {k-| U \cup X|-i \choose \ell-i} \sum_{T_i} \langle T_i\rangle_r=0.$$
Rewriting the above equation, we have
$$\sum_{i=0}^{\ell-1} (-1)^i {k-|U \cup X |-i \choose \ell-i} \sum_{T_i} \langle T_i\rangle_r = (-1)^{\ell +1} \sum_{T_{\ell}} \langle T_{\ell}\rangle_r$$

The terms on the left hand side of the above equation are contained in $\mbox{span}_K\{\langle A \rangle_r : A \in S(j), 0 \leq j \leq t \}$ by induction hypothesis since the sets $T_i$ have cardinality strictly less than $s$. We can rewrite the term on the right as
$$\sum_{T_{\ell}} \langle T_{\ell}\rangle_r = \sum_{T_{\ell} : m_{T_{\ell}} < m} \langle T_{\ell} \rangle_r + \sum_{T_{\ell}: m_{T_{\ell}}=m} \langle T_{\ell} \rangle_r$$

The first term on the right of the above equation belongs to $\mbox{span}_K\{\langle A \rangle_r : A \in S(j), 0 \leq j \leq t \}$ by induction hypothesis. Now, because $[U \cup X]= \sum_{T_{\ell}: m_{T_{\ell}}=m} \langle T_l \rangle_r$, we conclude that
\[
        [U \cup X] = \sum_{i=0}^{\ell-1} (-1)^{i+\ell+1} {r-| U \cup X |-i \choose \ell-i} \sum_{T_i} \langle T_i \rangle_k - \sum_{T_{\ell}: m_{T_{\ell}} < m} \langle T_{\ell} \rangle_k
\]
which proves Claim (i).

\vspace{0.1in}
\noindent{\bf Claim (ii).} For any $I \subset \{1,2, \ldots, 2m-1\}$ with $|I|=m$,
\[
        \sum_{U \subseteq I} (-1)^{|U|}[U \cup X] = 0.
\]

Claim (ii) can be proved as follows. By definition we have,
\begin{equation}\label{r_ecu4}
    \sum_{U \subseteq I} (-1)^{|U|}[U \cup X] = \sum_{U \subseteq I} (-1)^{|U|} \sum_{U \subseteq J} \langle J \cup X \rangle_r
\end{equation}

Consider any set $R \in {[n] \choose r}$. We want to count how many times the subset $R$ appears in the expression (\ref{r_ecu4}). We assume that $X \subseteq R$ and $|R \cap \{1,2,\ldots, 2m-1\}|$ is at least $m$ (otherwise, $R$ does not appear in (\ref{r_ecu4})). Define $\ell_1=|R \cap I|$ and $\ell_2= |(R \setminus I) \cap \{1, \ldots, 2m-1\}|$. We see that $R$ appears in (\ref{r_ecu4}) exactly
\[
    {\ell_1 \choose 0}{\ell_1+\ell_2 \choose m} - {\ell_1 \choose 1}{\ell_1+\ell_2-1 \choose m-1} + \cdots = \sum_{i=0}^m (-1)^i {\ell_1 \choose i}{\ell_1+\ell_2 -i \choose m-i}
\]
times. Now the sum on the right hand side of the above equation is equal to $0$ by the principle of inclusion and exclusion. This proves Claim (ii).

We will apply Claims (i) and (ii) to prove (\ref{extra}). By definition it is clear that $\langle B \rangle_r = [I \cup X]$. Hence, it follows from Claim (ii) that
\[
    \langle B \rangle_r = [I \cup X] = (-1)^{m+1} \sum_{U \subset I} (-1)^{|U|}[U \cup X]
\]
Therefore, (\ref{extra}) follows from Claim (i).

\end{proof}

\subsection{Proof of Theorem \ref{r_teo2}}

In this subsection we use the Bier bases to prove the resilience property of ranks of the higher inclusion matrices $W_{r,s}$ over an arbitrary field $K$. The following simple result from linear algebra will be needed.
\begin{lemma}\label{extra_lemma}
Let $u_1, \ldots, u_m$ be linearly independent vectors of a $K$-vector space $U$. Let $z_1, \ldots, z_m$ be vectors in $U$ such that $\mbox{span}\{u_1, \ldots, u_m\} \cap \mbox{span}\{z_1, \ldots, z_m\}=\{0\}$. Then $u_1+z_1, \dots, u_m + z_m$ are linearly independent vectors in $U$.
\end{lemma}

By definition of $\varphi_{s,r}$, it is easy to see that for $0 \leq s \leq r \leq n/2$, we have
\[
    \varphi_{s,r}(\langle A \rangle_s ) = {r-j \choose s-j} \langle A \rangle_r
\]
for every $A \in S(j)$ with $j=0,1,\ldots,s$. Therefore, the matrix of $\varphi_{s,r}$ with respect to the Bier basis $\{ \langle A \rangle_s : A \in S(j), 0 \leq j \leq s\}$ of $M^s$ and the Bier basis $\{ \langle A \rangle_r : A \in S(j), 0 \leq j \leq r\}$ of $M^r$ has a diagonal form. This proves that $\mbox{dim}_K(\mbox{im}(\varphi_{s,r}))$ is equal to
\[
    \sum_{j \in Y}  |S(j)| =\sum_{j \in Y} \left({n \choose j} - {n \choose j-1}\right)
\]
where $Y=\{j : 0 \leq j \leq s, {r-j \choose s-j} \neq_K 0\}$. This is precisely the $K$-rank formula given by Wilson \cite{n2} for the matrix $W_{r,s}$.

Let $S_n$ denote the symmetric group on $[n]$, and let $\sigma \in S_n$. For any $r$-subset $A$ of $[n]$ we define $\sigma(A)=\{\sigma(a): a \in A\}$. Similarly, if $\mathcal{F}$ is a family of $r$-subsets then $\sigma(\mathcal{F})=\{\sigma(A): A \in \mathcal{F}\}$.   The next lemma shows that we have a lot of freedom in the way we can remove rows from $W_{r,s}$ without lowering its $K$-rank.

\begin{lemma}\label{r_lemma2}
Assume that $0 \leq s < r \leq n/2$. Let $\mathcal{F}$ be a family of $r$-subsets of $[n]$. If there exist some $\sigma \in S_n$  such that $\sigma(\mathcal{F}^c) \subseteq S(r)$ then $\rank_{K}( W_{r,s}^{\mathcal{F}})=\rank_{K}(W_{r,s})$.
\end{lemma}

\begin{proof}
First, assume that $\mathcal{F}^c \subseteq S(r)$. We define the following linear map from $M^s$ to $M^r$
\[
 \varphi_{s,r}^{\mathcal{F}}(S)  = \sum_{S \subset R} R - \sum_{T \in \mathcal{F}^c,S \subset T} T, \mbox{ for all }S\subset [n], |S|=s, 
\]
where in the first sum $R$ runs over all $r$-subsets of $[n]$ containing $S$, and in the second sum $T$ runs over all $r$-subsets of $[n]$ containing $S$ such that $T \in \mathcal{F}^c$. It is clear from definition that $\dim_{K}(\mbox{im}\varphi_{s,r}^{\mathcal{F}}) = \rank_{K}(W_{r,s}^{\mathcal{F}})$.

Note that for every $j$-subset $A$ with $0 \leq j \leq s$ and $\mbox{rk}(A)=j$ we have
\begin{equation}\label{r_ecu5}
\varphi_{s,r}^{\mathcal{F}}(\langle  A \rangle_s) = {r-j \choose s-j} \langle  A \rangle_r - \sum_{T \in \mathcal{F}^c, A \subset T } {r-j  \choose s-j} T
\end{equation}
Recall that by assumption $\mathcal{F}^c \subseteq S(r)$, so any $T \in \mathcal{F}^c$ is actually a basis element of the Bier basis of $M^r$. Thus the matrix of $\varphi_{s,r}^{\mathcal{F}}$ with respect to the Bier bases of $M^r$ and $M^s$ is almost diagonal.

Let $W$ be the subspace of $M^r$ spanned by the following set of linearly independent vectors $$\left\{ {r-j \choose s-j} \langle  A \rangle_r : \mbox{ } A \in S(j), j \in Y \right\},$$
where $Y= \{j : 0 \leq j \leq s,  {r-j \choose s-j}\neq_K 0\}$. It is clear from the definition of the Bier basis of $M^r$ that
$$ W \cap  \mbox{span} \left\{ \sum_{T \in \mathcal{F}^c, A \subset T } {r-j  \choose s-j} T  : \mbox{ } A \in S(j), j \in Y \right\}=\{0\}.$$
Therefore, by Lemma \ref{extra_lemma} and (\ref{r_ecu5}) we conclude that the vectors in
\[
\bigcup_{j \in Y} \left\{ \varphi_{s,r}^{\mathcal{F}} (\langle A \rangle_s ) : \mbox{ } A \in S(j) \right\}
\]
are linearly independent. This implies that 
\[
   \dim_{K}(\mbox{im}\varphi_{s,r}^{\mathcal{F}}) \geq \sum_{j \in Y} \left({n \choose j } - {n \choose j-1}\right)
\]
Hence, Lemma \ref{r_lemma2} follows from the trivial upper bound $\rank_{K} (W_{r,s}^{\mathcal{F}}) \leq \rank_{K} (W_{r,s})$ and Wilson's rank formula.

Now, if $\mathcal{F}^c \nsubseteq S(r)$ then by assumption there exists $\sigma \in S_n$ such that $\sigma(\mathcal{F}^c) \subseteq S(r)$. We use $\sigma$ to define the following invertible linear transformations,
\[
\begin{array}{ccccc}
\Phi_r^{\sigma}: & M^r & \rightarrow & M^r\\
                          & R     & \mapsto & \sigma(R) 			
\end{array},
\quad
\begin{array}{ccccc}
\Phi_s^{\sigma}: & M^s & \rightarrow & M^s\\
                          & S     & \mapsto & \sigma(S) 			
\end{array}
\]

From the above definitions it follows that
\[
\varphi_{s,r}^{\mathcal{F}}= (\Phi_r^{\sigma})^{-1} \circ \varphi_{s,r}^{\sigma(\mathcal{F})} \circ \Phi_s^{\sigma}
\]
Thus,  $\dim_{K}(\mbox{im}\varphi_{s,r}^{\mathcal{F}}) =\dim_{K}(\mbox{im}\varphi_{s,r}^{\sigma(\mathcal{F})})$. The proof of Lemma \ref{r_lemma2} is now complete. 
\end{proof}

Now, we apply Lemma \ref{r_lemma2} to prove Theorem \ref{r_teo2}.
\begin{proof}[\bf Proof of Theorem \ref{r_teo2}] In order to apply Lemma~\ref{r_lemma2}, we will show that when $|\mathcal{F}^c|\leq \frac{n-1}{r}$, it is always possible to find $\sigma\in S_n$ such that $\sigma(\mathcal{F}^c)\subseteq S(r)$. First, note that the case where $n=2r$ is completely trivial because in that case $\mathcal{F}^c$ contains at most one $r$-subset by the assumption that $|\mathcal{F}^c|\leq \frac{n-1}{r}$. Secondly, note that it is enough to prove Theorem \ref{r_teo2} for all $n$ of the form  $\alpha r + 1$, with $\alpha \geq 2$. In fact, the result for other values of $n$ follows immediately from the result in the cases where $n$ is of the form $\alpha r +1$, $\alpha\geq 2$.

Recall that an $r$-subset $A$ of $[n]$ is in $S(r)$ if and only if the path associated with $A$ does not cross the main diagonal; this latter condition in turn is equivalent to the following:  for every $i=1, \ldots, 2r$ we have that $|A \cap [i]| \leq \lfloor \frac{i}{2} \rfloor$.

Given $n=\alpha r + 1$ the critical case occurs when $\mathcal{F}^c$ consists of $\alpha$ disjoint $r$-subsets. Even in this case there exists $\sigma \in S_n$ such that $\sigma(\mathcal{F}^c) \subseteq S(r)$. For example,  for $\alpha=2$ it is possible to map the two r-subsets in $\mathcal{F}^c$  to the $r$-subsets $\{2,4,6,\ldots,2r\}$ and $\{3,5,7,\ldots, 2r+1\}$ which are contained in $S(r)$. The conclusion of the theorem now follows from Lemma~\ref{r_lemma2}.
\end{proof}





\section{Rank Resilience: the Vector Space Case}

The goal of this section is to prove Theorem 6. Throughout this section, $V$ is an $n$-dimensional vector space over $\mathbb{F}_q$, where $q=p^t$ is a prime power.

\subsection{The $GL(n,q)$-module $M_q^r$}

In this section, we assume that $K$ is a field of characteristic coprime to $q=p^t$, containing a primitive $p^{\rm th}$ root of unity. For every $0 \leq r \leq n$, we denote by $M_q^r$ the $K$-vector space spanned by the $r$-dimensional subspaces of $V$. Hence, the set of $r$-dimensional subspaces forms a ``canonical" basis of $M_q^r$. 

Let $GL(n,q)$ be the group of all invertible linear transformations from $V$ to $V$. Each element of $GL(n,q)$ induces a permutation on the set of $r$-dimensional subspaces of $V$. Thus, $M_q^r$ is a $GL(n,q)$-permutation module for $0 \leq r \leq n$.

The Specht module $S^{(n-r,r)}$ is the submodule of $M_q^r$ defined by
\[
S^{(n-r,r)}= \bigcap_{0\leq j<r} \left\{ \ker \phi : \phi \in \mbox{Hom}_{GL(n,q)}(M_q^r, M_q^j)  \right\},
\]
where $\mbox{Hom}_{GL(n,q)}(M_q^r, M_q^j)$ is the set of all $GL(n,q)$-module homomorphisms from $M_q^r$ to $M_q^j$. We remark that the Specht modules $S^{(n-r,r)}$ over the complex are irreducible; for $K$ of positive characteristics, the Specht modules are not necessarily irreducible. In \cite{n11}, James proved that the dimension of $S^{(n-r,r)}$ over $K$ is equal to ${n \brack r} - {n \brack r-1}$. He also proved the following important result about  Specht modules.

\begin{theorem}\label{M_submodule_teo1}
{\em (The Submodule Theorem)} Let $\langle \cdot, \cdot \rangle$ be the inner product on $M_q^r$ such that for any two $r$-dimensional subspaces $X, Y$ of $V$ we have that $\langle X, Y \rangle= 1$ if $X=Y$ and $0$, otherwise. If $W$ is a $GL(n,q)$-submodule of $M_q^r$ then either $S^{(n-r,r)} \subseteq W$ or $W \subseteq (S^{(n-r,r)})^{\perp}$, where $(S^{(n-r,r)})^{\perp}$ is the orthogonal complement of $S^{(n-r,r)}$ with respect to $\langle \cdot, \cdot \rangle$.
\end{theorem}

Recently, Brandt et al.  \cite{n9} found a basis of $S^{(n-r,r)}$ which is indexed by the standard tableaux of shape $(n-r,r)$. We will recall some definitions and results from \cite{n9}  to describe this ``standard basis''. 

Let $0\leq r \leq n-r$. Consider a rectangular $r\times (n-r)$ array of boxes, which are depicted in the following figure.
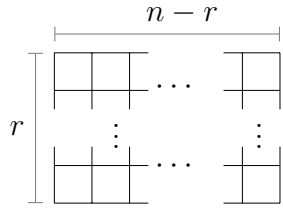
\begin{center}
\begin{tikzpicture}[scale=0.5]
     \draw[very thin] (0,0) grid (2.5,1.5);
     \node (n1) at (1.5,2) {{ $\vdots$}};
     \node (n1) at (3.1,1) {{ $\cdots$}};
     \node (n1) at (3.1,3.1) {{ $\cdots$}};
     \draw[very thin] (0,2.5) grid (2.5,4);
     \draw[very thin] (4.5,0) grid (6,1.5);
     \node (n1) at (5.3,2) {{ $\vdots$}};
     \draw[very thin] (4.5,2.5) grid (6,4);
     \draw[|-|, color=gray](0,4.5) -- (6,4.5) node[above, xshift=-1.3cm, color=black] {$n-r$};
     \draw[|-|, color=gray](-0.5,0) -- (-0.5,4) node[left, yshift=-1cm, color=black] {$r$};
\end{tikzpicture}
\end{center}

It is well known that every $r$-subset $A$ of $[n]$ corresponds to a path connecting the top left corner with the right bottom corner of the above array of boxes. Specifically, the $i$-th step is $S$ (south) or $E$ (east) according as $i \in A$ or $i \notin A$. For example, the $r$-subsets contained in $S(r)$ correspond to the paths that do not cross the main diagonal of the array of boxes. We denote by $P(n-r,r)$ the set of all paths connecting the top left with the bottom right corner of an $r\times (n-r)$ array of boxes. Then by the correspondence described above, $|P(n-r, r)|={n \choose r}$.

\begin{example}
Let $n=5$ and $r=2$. Consider the path marked in red in the following figure.
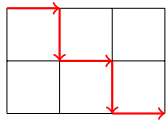
\begin{center}
\begin{tikzpicture}[scale=0.7]
     \draw[very thin] (0,0) grid (3,2);
     \draw[->, color=red, thick](0,2) -- (1,2);
     \draw[->, color=red, thick](1,2) -- (1,1);
     \draw[->, color=red, thick](1,1) -- (2,1);
     \draw[->, color=red, thick](2,1) -- (2,0);
     \draw[->, color=red, thick](2,0) -- (3,0);
\end{tikzpicture}
\vspace{0.3cm}
\end{center}
The path, denoted by $\pi$, is $ESESE$ where $E$ stands for east and $S$ for south. Hence, the $2$-subset of $[5]$ corresponding to $\pi$ is $\{2,4\}$.
\end{example}

We impose the reverse lexicographic order on the set $P(n-r,r)$ of paths. For example, the elements of $P(2,2)$ are ordered in the following way:
\[
	SSEE < SESE < SEES < ESSE < ESES < EESS.
\]

Given any path $\pi \in P(n-r,r)$ we can fill the boxes below $\pi$ by using elements from $\mathbb{F}_q$, and we use $c(\pi)$ to denote the number of such fillings. For example, for $n=7$ and $r=3$,
\[
\begin{ytableau}
a_1 &   &   &  \\
a_2 &  a_3 &   & \\
a_4 &  a_5 & a_6  &
\end{ytableau}
\]
where $a_i \in \mathbb{F}_q$, $\pi=ESESESE$, and $c(\pi)=q^6$. The following well-known result establishes a bijection between these objects and the $r$-dimensional subspaces of $V$.  A proof can be found in \cite{n9}. 

\begin{lemma}\label{resilience_lemma1} {\em (Brandt et al. \cite{n9})}
Choosing a path $\pi \in P(n-r,r)$ and then filling the boxes below the path with elements of $\mathbb{F}_q$ is a way of encoding an $r$-dimensional subspace of $V$. Every such subspace can be uniquely encoded in this way.
\end{lemma}

The above lemma shows that $\sum_{\pi \in P(n-r,r)}c(\pi)={n\brack r}$. The proof of Lemma \ref{resilience_lemma1} associates the reduced echelon form of a subspace to a path $\pi$ and a filling for that path. For example, if a $3$-dimensional vector subspace of $\mathbb{F}_q^{7}$ has the following reduced echelon form
\[
    \left(
      \begin{array}{ccccccc}
        a & 1 & 0 & 0 & 0 & 0 & 0 \\
        b & 0 & 1 & 0 & 0 & 0 & 0 \\
        c & 0 & 0 & d & 1 & 0 & 0 \\
      \end{array}
    \right)
\]
then the path and filling corresponding to this vector subspace is,
\[
\begin{ytableau}
a &    &   & \\
b &    &   & \\
c & d  &   &
\end{ytableau}
\quad \mbox{ with }\pi=ESSESEE.
\]
Note that here the steps where $\pi$ makes a SOUTH move correspond to the columns which contain a leading one in the reduced echelon form of the 3-dimensional subspace. For any $r$-subspace $X$ of $V$ we will denote by $\pi(X)$ the path corresponding to $X$.

\begin{definition} (Brandt et al. \cite{n9})
Suppose that $v \in M_q^r$, and write
\[
v = \sum_{ X \in {V \brack r} } c_X X, \quad \mbox{ where } c_X \in K.
\]
\begin{enumerate}
	\item For each path $\pi\in P(n-r,r)$, let 
		\[
			v(\pi)=\sum_{X: \pi(X)=\pi} c_X X.
		\]
	\item If $v \neq 0$, then let $\mbox{greatest}(v)$ denote the greatest\footnote{Greatest with respect to the reverse lexicographic order imposed on $P(n-r,r)$} path $\pi \in P(n-r,r)$ such that $v(\pi) \neq 0$.
	\item If $v \neq 0$, then let $\mbox{top}(v)=v({\rm greatest}(v))$.
	\item If $U$ is a subspace of $M_q^r$ and $\pi \in P(n-r,r)$, then let
	\[
		U(\pi)=\{u(\pi): 0 \neq u \in U \mbox{ and } {\rm greatest}(u)=\pi \} \cup \{0\}.
	\]
\end{enumerate}
\end{definition}

A couple of remarks are in order. First note that for any $\pi\in P(n-r,r)$, we have $M_q^r(\pi)=\{\sum_{X: \pi(X)=\pi}c_X X\mid c_X\in K\}$. Secondly, we have
$$M_q^r=\bigoplus_{\pi\in P(n-r,r)} M_q^r(\pi)$$

Let $\theta$ be an additive character of $\mathbb{F}_q$ \cite{Serre}. Suppose that $X$ and $L$ are $r$-dimensional subspaces of $V$ such that $\pi(X)=\pi(L)$. Let $\chi_L$ be the linear character on $M_q^r$ defined by
\[
		\chi_L(X)= \prod_{i=1}^r \prod_{j=1}^{n-r} \theta(l_{i,j}x_{i,j})
\]
where $l_{i,j}$ and $x_{i,j}$ denote the $(i,j)$-entries in the filling corresponding to $L$ and $X$, respectively (here we are assuming that the boxes above the path are filled with zeros).  Using the character $\chi_L$ we define the following element of $M_q^r$
\[
e_L=\sum_{X: \pi(X) = \pi(L) } \chi_{L}(-X)X
\]
for every $L \in  {V \brack r}$. Furthermore, the orthogonality relations for linear characters imply that the sets
\[
\left\{ e_L : L \in {V \brack r} \right\} \quad \mbox{ and } \quad \left\{ e_L : L \in {V \brack r} \mbox{ with }\pi(L)=\pi \right\}
\]
form a basis of $M_q^r$ and $M_q^r(\pi)$, respectively.

\begin{definition}\label{M_submodule_def1} (Brandt et al. \cite{n9})
Let $\pi \in P(n-r,r)$ be a path connecting the top left with the bottom right corner of an array of boxes of size $r$ by $n-r$. Label the corners of the array by ordered pairs $(i,j)$ with $i=1, \ldots, r+1$ and $j=1, \ldots, n-r+1$. For every corner $(i,j)$, we define $r(i,j)=j-i$. Let $X$ be an $r$-dimensional subspace of $V$ such that $\pi(X)=\pi$. We say that $X$ is {\it good} if its associated filling of the boxes to the south of $\pi$ with elements of $\mathbb{F}_q$ is {\it good}: for each corner $(i,j)$ through which the path $\pi$ passes, the matrix with bottom left and top right corners having coordinates $(r+1,1)$ and $(i,j)$, respectively, has rank at most $r(i,j)$. If $X$ is not good then we say it is {\it bad}.
\end{definition}

Note that by Definition \ref{M_submodule_def1} if a path $\pi \in P(n-r,r)$ crosses the main diagonal of the array of boxes (that is, the $r$-subset corresponding to $\pi$ does not belong to $S(r)$) then there is no good $r$-dimensional subspace $X$ with $\pi(X)=\pi$. The reason is simple: If $\pi\in P(n-r,r)$ crosses the main diagonal, then there is a corner $(i,j)$, with $i>j$, through which $\pi$ passes; for that corner, we have $r(i,j)=j-i<0$; hence there is no good filling below the path $\pi$. It follows that if $L$ is a good $r$-dimensional subspace of $V$ then $\pi(L) \in S(r)$. The next theorem gives a ``standard basis" for the Specht module $S^{(n-r,r)}$.

\begin{theorem}\label{M_submodule_base_teo} {\em  (Brandt et al. \cite{n9})}
For each good $r$-dimensional subspace $L$ of $V$ there exists a vector $z_L \in M_q^r$ with $top(z_L)=e_L$ such that $z_L$, with $L$ running through the set of good $r$-dimensional subspaces of $V$,  form a basis of $S^{(n-r,r)}$.
\end{theorem}
As was remarked earlier, every path $\pi \in P(n-r,r)$ that does not cross the main diagonal corresponds to a unique $r$-subset in $S(r)$. Thus, by abuse of notation we will denote also by $S(r)$ the set of paths that do not cross the main diagonal. 
Since the elements of $S(r)$ are in one-to-one correspondence with the standard tableaux of shape $(n-r,r)$, it follows that Theorem \ref{M_submodule_base_teo} provides a basis of $S^{(n-r,r)}$ which is indexed by the standard tableaux of shape $(n-r,r)$; that is  the reason why the basis in Theorem~\ref{M_submodule_base_teo} is called a standard basis. 

To prove Theorem \ref{r_teo3} we will need to introduce another submodule of $M_q^r$. For $0\leq j\leq r$, define the linear transformation $\varphi_{j,r} : M^j_q \rightarrow M^r_q$ as follows. For any $j$-dimensional subspace $X$ of $V$, define
\[
\varphi_{j,r}(X) = \sum_{X \subseteq R} R,
\]
where the sum runs over all the $r$-dimensional subspaces containing $X$; the definition of $\varphi_{j,r}$ is then extended to all elements of $M_q^j$ by linearity. We remark that $\varphi_{j,r}$ is not only a linear map, but also a $GL(n,q)$-module homomorphism from $M_q^j$ to $M_q^r$ since for any $g \in GL(n,q)$ we have $g \cdot \varphi_{j,r} = \varphi_{j,r} \cdot g$. To simplify notation, for any $j$-dimensional subspace $X$ of $V$, with $j \leq r$, we denote by $\langle X \rangle_r$ the image of $X$ under $\varphi_{j,r}$.   Note that the subspace inclusion matrix $W_{r,j}(q)$ is the matrix of $\varphi_{j,r}$ with respect to the canonical bases of $M_q^j$ and $M_q^r$. It follows from the results in Frumkin and Yakir \cite{n10} that
\begin{equation}\label{vs_ecu3}
\dim_{K}(\mbox{im}(\varphi_{j,r})) =  \sum_{i\in Y } \left({n \brack i} - {n \brack i-1}\right),
\end{equation}
where $Y=\{i: 0 \leq i \leq j, {r-i \brack j-i} \neq_K 0 \}$. Consider the following subspace of $M_q^r$,
\[
	U_{r-1}=  \varphi_{0,r}(M_q^0) + \varphi_{1,r}(M_q^1) + \cdots + \varphi_{r-1,r}(M_q^{r-1}).
\]
That is $U_{r-1}$ is the column space of 
$$
\left[ W_{r,0}(q) \mid W_{r,1}(q) \mid \cdots \mid W_{r,r-1}(q)  \right].
$$
Note that $U_{r-1}$ is a $GL(n,q)$-submodule of $M_q^r$. This module was studied by Frumkin and Yakir \cite{n10}, in which it was shown that the dimension over $K$ of $U_{r-1}$ is ${n \brack r-1}$.

\subsection{Proof of Theorem \ref{r_teo3}}

In this subsection we will give the proof of Theorem \ref{r_teo3}. Our approach will be similar to the one used in the proof of Theorem \ref{r_teo2}. However, since we do not have a $q$-analogue of the Bier basis of $M_q^r$ we will use the results from representation theory that were introduced in Section 3.1.

For $\pi \in P(n-r,r)$, define the {\it leading term} of $\pi$ to be the number of $E$ moves before the first $S$ move. We define
$$S(r)^{-}=\{\pi\in P(n-r,r): {\rm the}\; {\rm  leading}\; {\rm term}\; {\rm of}\; \pi<r\},$$
and
$$S(r)^{+}=\{\pi\in P(n-r,r): {\rm the}\; {\rm  leading}\; {\rm term}\; {\rm of}\; \pi\geq r\}.$$
From definition we have 
\[
	S(r)= S(r)^{-}\; \dot\cup\; S(r)^{+}
\]
Also, if $\pi\in S(r)^{+}$, every filling of $\pi$ is good since for any corner $(i,j)$ through which $\pi$ passes, $r(i,j)$ is automatically greater than or equal to the rank of the matrix with bottom left and top right corners having coordinates $(r+1,1)$ and $(i,j)$. As a preparation, we first prove the following lemma.

\begin{lemma}\label{vs_lem_tech1}
Let $K$ be a field  of characteristic coprime to $q=p^t$ and containing a primitive $p^{\rm th}$ root of unity. With notation as above, we have
\[
	U_{r-1} \cap \bigoplus_{\pi \in S(r)^{+}} M_q^r(\pi)= \{0\}.
\]
\end{lemma}

\begin{proof}
We will use the inner product $\langle \cdot, \cdot \rangle$ defined on $M_q^r$ given in Theorem~\ref{M_submodule_teo1}. Since ${\rm dim}(S^{(n-r,r)}) > {\rm dim}(U_{r-1})$, we see by the Submodule Theorem that $U_{r-1}$ is contained in $(S^{(n-r,r)})^{\perp}$.  Thus, for any $z \in S^{(n-r,r)}$ and any $v \in U_{r-1}$ we have $\langle z, v\rangle=0$.

As we remarked above any $r$-dimensional subspace $L$ of $V$ with $\pi(L) \in S(r)^{+}$ is good. Therefore, if $\pi \in S(r)^{+}$ then the vectors in the set $\{ e_L :  L \mbox{ good and }\pi(L)=\pi\}$ form a basis of $M_q^r(\pi)$. Combining this fact with Theorem \ref{M_submodule_base_teo}, we conclude that the Specht module $S^{(n-r,r)}$ contains a vector $w_L$ such that $top(w_L)=L$ for each $L$ with $\pi(L) \in S(r)^{+}$.

Given any vector $v \in M_q^r$ we can use the canonical basis of  $M_q^r$ to represent $v$ as a column vector (that is, we index the coordinates of the column vector by $r$-subspaces of $V$). We arrange the canonical basis with respect to the reverse lexicographic order. Therefore, on the top we have the subspaces associated to the paths in $S(r)^{+}$, then the subspaces whose associated paths are in $S(r)^{-}$, and finally the ones associated to paths in $P(n-r,r) \setminus S(r)$.   

Now, given an arbitrary basis of $U_{r-1}$, we consider the basis elements represented as column vectors with respect to the canonical basis.  Applying column operations to the basis vectors we can get a new basis of $U_{r-1}$ in reduced echelon form such that the leading ones appear from left to right and from the bottom to the top. 

We claim that no leading ones of this new basis appear on a row indexed by a subspace $L$ with $\pi(L) \in S(r)^{+}$. Note that this is enough to prove the conclusion of the lemma.

To prove our claim we proceed by contradiction. Suppose that after column operations one of the basis vectors $v'$ of $U_{r-1}$ has a leading one in a row indexed by a subspace $L$ with $\pi(L) \in S(r)^{+}$. Then,  $\langle v', w_L \rangle=1$ which is a contradiction because $U_{r-1} \subseteq (S^{(n-r,r)})^{\perp}$.
\end{proof}

Now we prove a vector space analogue of Lemma \ref{r_lemma2}.  To state the result we introduce some notation. For any $g \in GL(n,q)$ and any family $\mathcal{F}$ of $r$-subspaces of $V$ we denote by $g(\mathcal{F})$ the family of $r$-subspaces $\{ g(X): X \in \mathcal{F} \}$. Furthermore, consider the following set of $r$-dimensional vector subspaces of $V$:
\[
S(r)^{+}_q = \left\{ X \in {V \brack r} : \pi(X) \in S(r)^{+} \right\}
\]
That is, $S(r)^{+}_q$ is the set of $r$-subspaces of $V$ whose associated paths are in $S(r)^{+}$.

\begin{lemma}\label{vs_resilience}
Suppose $0 \leq s < r \leq n/2$. Let $\mathcal{F}$ be a family of $r$-dimensional subspaces of $V$ and $K$ a field with $\mbox{char}(K)\neq p$. If there exists $g \in GL(n,q)$ such that $g(\mathcal{F}^c) \subseteq S(r)_q^{+}$ then
\[
\rank_{K}(W_{r,s}^{\mathcal{F}}(q)) = \rank_{K}( W_{r,s}(q)).
\] 
\end{lemma}

\begin{proof}
Without loss of generality we may assume that $K$ contains a primitive $p^{\rm th}$ root of unity. Indeed, if $K$ does not contain a primitive $p$-th root of unity then we can extend $K$ to a larger field and this does not change the rank of the matrices $W_{r,s}(q)$ or $W_{r,s}(q)^{\mathcal{F}}$.

First, assume  that  $\mathcal{F}^c \subseteq S(r)_q^{+}$. Consider the following subspaces of $M_q^s$,
\begin{equation}\label{vs_ecu1}
	W_j= \varphi_{0,s}(M_q^0) + \varphi_{1,s}(M_q^1) + \cdots + \varphi_{j,s}(M_q^{j}).
\end{equation}
for $j=0,1,\ldots ,s$. It is clear that 
\begin{equation}\label{vs_ecu2}
	W_0 \subset W_1 \subset \cdots \subset W_s
\end{equation}
Furthermore, the dimension of $W_j$ over $K$ was shown to be ${n \brack j}$ in \cite{n10}.  Therefore, it follows from equations (\ref{vs_ecu1}) and (\ref{vs_ecu2}) that $M_q^s$ has a basis with the following property: For each $j$ from $0$ to $s$, ${n \brack j} - {n \brack j-1}$ of the elements of the basis are of the form $\langle  X \rangle_s$ with $X \in {V \brack j}$. For $j=0,1,\ldots ,s$, we denote by $B_j$ a set of $j$-dimensional subspaces of $V$ with cardinality ${n \brack j} - {n \brack j-1}$ chosen in such a way that 
\[
	\bigcup_{j=0}^s \left\{  \langle X \rangle_s: X \in B_j \right\}
\]
is a basis of $M_q^s$.

By the definition of $\varphi_{s,r}$ and straightforward computations, we have 
\begin{equation}\label{vs_ecu4}
\varphi_{s,r}( \langle X \rangle_s) = {r-j \brack s-j} \langle X \rangle_r
\end{equation}
for all $X \in B_j$ with $j$ ranging from $0$ to $s$. 

Let $Y_s=\{ j : 0 \leq j \leq s \mbox{ such that } {r-j \brack s-j} \neq_K 0 \}$ and $Z_s=\{ j : 0 \leq j \leq s \mbox{ such that } {r-j \brack s-j} =_K 0 \}$. Equations (\ref{vs_ecu3}) and (\ref{vs_ecu4}) imply that the set   $$\bigcup_{j \in Z_s} \left\{  \langle X \rangle_s: X \in B_j \right\}$$ forms a basis of the kernel of $\varphi_{s,r}$. Therefore, the set
\begin{equation}\label{vs_ecu5}
	\bigcup_{j \in Y_s} \left\{  \langle X \rangle_r: X \in B_j \right\}
\end{equation}
forms a basis for the image of $\varphi_{s,r}$; so in particular these vectors are linearly independent in $M_q^r$.

Now, we proceed in the same way as in the proof of Lemma \ref{r_lemma2}.  Consider the following linear transformation from $M^s_q$ to $M^r_q$
\[
 \varphi_{s,r}^{\mathcal{F}^c}(S)  = \sum_{S \subseteq R} R - \sum_{T \in \mathcal{F}^c,S \subseteq T} T
\]
where $R$ runs over all $r$-dimensional subspaces of $V$ containing $S$, and $T$ runs over all $r$-dimensional subspaces of $V$ containing $S$ such that $T \in \mathcal{F}^c$. It is clear from definition that $\dim_{K}(\mbox{im}\varphi_{s,r}^{\mathcal{F}^c}) = \mbox{rank}_{K} W_{r,s}^{\mathcal{F}}(q)$. Furthermore, note that for every $X \in B_j$ with $0 \leq j \leq s$ we have 
\begin{equation*}
\varphi_{s,r}^{\mathcal{F}^c}(\langle  X \rangle_s) = {r-j \brack s-j} \langle  X \rangle_r - \sum_{T \in\mathcal{F}^c, X \subseteq T } {r-j  \brack s-j} T.
\end{equation*}

Note that the vectors in
$$  \left\{ {r-j \brack s-j} \langle  X \rangle_r  :  X \in B_j \mbox{ with } 0 \leq j \leq s \right\}  $$
are linearly independent. Moreover, for every $X \in B_j$ with $0 \leq j \leq s$, the vector $\displaystyle \sum_{T \in\mathcal{F}^c, X \subseteq T } {r-j  \brack s-j} T$ is contained in $U_{r-1}$. Therefore, it follows from Lemma \ref{extra_lemma} and Lemma \ref{vs_lem_tech1}  that the vectors in
\[
	\bigcup_{j \in Y} \left\{  \varphi_{s,r}^{\mathcal{F}^c}(\langle  X \rangle_s)   : X \in B_j \right\}
\]
are linearly independent in $M_q^r$. Therefore 
\[
   \sum_{j \in Y} \left({n \brack j } - {n \brack j-1}\right) \leq \dim_{K}(\mbox{im}\varphi_{s,r}^{\mathcal{F}^c})
\]
Hence, Lemma \ref{vs_resilience} follows from the trivial upper bound $\mbox{rank}_KW_{r,s}^{\mathcal{F}}(q) \leq \mbox{rank}_KW_{r,s}(q)$ and the $q$-analogue of Wilson's rank formula for $W_{r,s}(q)$.

Now, if $\mathcal{F}^c \nsubseteq S(r)_q^{+}$, then by assumption there exists $g \in GL(n,q)$ such that $g(\mathcal{F}^c) \subseteq S(r)_q^{+}$. As in the proof of Lemma \ref{r_lemma2}, we can use $g$  to define the following invertible linear transformations,
\[
\begin{array}{ccccc}
\Phi_r^{g}: & M^r_q & \rightarrow & M^r_q\\
                          & R     & \mapsto & g(R) 			
\end{array},
\quad
\begin{array}{ccccc}
\Phi_s^{g}: & M^s_q & \rightarrow & M^s_q\\
                          & S     & \mapsto & g(S) 			
\end{array}
\]

From the above definitions, it follows that
\[
\varphi_{s,r}^{\mathcal{F}^c}= (\Phi_r^{g})^{-1} \circ \varphi_{s,r}^{g(\mathcal{F}^c)} \circ \Phi_s^{g}
\]
Hence $\dim_{K}(\mbox{im}\varphi_{s,r}^{\mathcal{F}^c}) =\dim_{K}(\mbox{im}\varphi_{s,r}^{g(\mathcal{F}^c)})$. The proof of the lemma is now complete.

\end{proof}

In the statement of the following corollary, for an $r$-dimensional subspace $X$ of $V$, we denote also by $\pi(X)$ the unique $r$-subset of $[n]$ corresponding to the path in $P(n-r,r)$ associated with $X$.

\begin{corollary}\label{vs_cor1}
Suppose that $0 \leq s < r \leq n/2$. Let $\mathcal{F}$  be a family of $r$-subspaces of $V$ satisfying that
\begin{equation}\label{resilience_ecu6}
	\left| \cup_{X \in \mathcal{F}^c } \pi(X) \right| \leq n-r
\end{equation}
Then $\rank_{K}(W_{r,s}^{\mathcal{F}}(q)) = \rank_{K}( W_{r,s}(q))$.
\end{corollary}

\begin{proof}
By Lemma \ref{vs_resilience} it is enough to show that there exists $g \in GL(n,q)$ such that $g(\mathcal{F}^c) \subseteq S(r)_q^{+}$. Recall that every $r$-dimensional subspace of $V$ can be represented by a unique $r$ by $n$ matrix in reduced echelon form. The condition $\left| \bigcup_{X \in \mathcal{F}^c } \pi(X) \right| \leq n-r$ implies that there are at least $r$ columns that do not contain a leading one for any of the subspaces in $\mathcal{F}^c$. Let $i_1 < i_2 < \cdots < i_l$ be the indices of the columns corresponding to the leading ones of all subspaces in $\mathcal{F}^c$. By assumption we have that $l \leq n-r$; so there exists a permutation sending $i_l \rightarrow n, i_{l-1} \rightarrow n-1, \ldots, i_1 \rightarrow n-l+1$ where $n-l+1 > r$.

This implies that there exists a linear transformation $g\in GL(n,q)$ sending every $X \in \mathcal{F}^c$ to a subspace $g(X)$ such that none of the leading ones of the reduced echelon form of $g(X)$ appears in the first $r$ columns; hence 
$g(X) \in S(r)^{+}_q$ for every $X \in \mathcal{F}^c$. The proof of the corollary is now complete.
\end{proof}

Theorem~\ref{r_teo3} is an immediate consequence of Corollary \ref{vs_cor1} because any family of $r$-subspaces ${\mathcal F}$ of $V$ satisfying that $|\mathcal{F}^c| \leq \frac{n}{r}-1$ clearly satisfies (\ref{resilience_ecu6}).

\section{Concluding Remarks}

In this paper, we have proved two variations of Keevash's result (Theorem \ref{r_teo1}). First, we show that the rank of the subset-inclusion matrix $W_{r,s}^{\mathcal{F}}$ is resilient over any field. More precisely, if a family $\mathcal{F}$ of $r$-subsets of $[n]$ satisfies the condition that $|{\mathcal F}^c|\leq \frac{n-1}{r}$, then $\mbox{rank}_{K}(W_{r,s}^{\mathcal F}) = \mbox{rank}_{K}( W_{r,s})$ for any field $K$. Note that a less restrictive bound on $|{\mathcal F}^c|$ was obtained in \cite{n7} when $K$ is a field of characteristic zero. More precisely, if ${\rm char}(K)=0$ and $n$ is large, it was shown in \cite{n7} that $\mbox{rank}_{K}( W_{r,s}^{\mathcal{F}})=\mbox{rank}_{K}(W_{r,s})$ for all families ${\mathcal F}$ of $r$-subsets of $[n]$ satisfying that $|\mathcal{F}^c| < {n-s \choose r-s}$. Therefore the following question arises naturally:  does Theorem \ref{r_teo2} remain true under the assumption that $|\mathcal{F}^c| < {n-s \choose r-s}$?

Secondly, we prove a $q$-analogue of Theorem \ref{r_teo1}: If the size of a family $\mathcal{F}$ of $r$-dimensional subspaces of $\mathbb{F}_{q}^n$ is close enough to ${n \brack r}$ then $\rank_{K}(W_{r,s}^{\mathcal F}(q)) = \rank_{K}( W_{r,s}(q))$ for any field $K$ of characteristic coprime to $q$. 

The condition in Theorem \ref{r_teo3} on the size of $\mathcal{F}^c$ is somewhat restrictive. For example, if we remove all the $r$-subspaces containing a particular $s$-subspace the rank over the rationals of  $W_{r,s}^{\mathcal{F}}(q)$ has to decrease at least by one. So a natural question is: does Theorem \ref{r_teo3} remain true under the assumption  $|\mathcal{F}^c| < {n-s \brack r-s}$?


\begin{thebibliography}{}

\bibitem{babaif} L. Babai, P. Frankl, {\em Linear Algebraic Methods in Combinatorics}, Preliminary version 2, 1992.

\bibitem{n1} 
T. Bier.
Remarks on recent formulas of Wilson and Frankl,
\emph{Europ. J. Combin.} {\bf 14} (1993),1--8.

\bibitem{n9} 
M. Brandt, R. Dipper, G. James and S. Lyle.
Rank polynomials, 
\emph{Proc. London Math. Soc.}  {\bf 98} (2009), 1--18.

\bibitem{Frankl2}
    P. Frankl,
     \newblock Erd\H{o}s-Ko-Rado theorem with conditions on the maximal degree,
     \newblock {\em J. Combin. Theory} (A) {\bf 46} (1987), 252--263.   

\bibitem{n3} 
P. Frankl,
Intersection theorems and mod p rank of inclusion matrices,
\emph{J. Combin. Theory} (A) {\bf 54} (1990), 85--94.


\bibitem{n20} K. Friedl, L. Ronyai, Order shattering and Wilson's theorem, {\em Disc. Math.} {\bf 270} (2003), 127--136.

\bibitem{n10} 
A. Frumkin and A. Yakir,  
Rank of inclusion matrices and modular representation theory,
\emph{Israel J. Math.} {\bf 71} (1990), 309--320.
  

\bibitem{godsil} C. D. Godsil, Problems in algebraic combinatorics, {\em Electronic J. Combin.} {\bf 2} (1995), Feature 1, approx. 20 pp. (electronic).


\bibitem{n4} 
D. H. Gottlieb,
A certain class of incidence matrices,
\emph{Proc. Amer. Math. Soc.}  {\bf 17} (1966),1233-1237.
 
\bibitem{n7} 
C. Grosu, Y. Person, and T. Szab\'o, 
On the rank of higher inclusion matrices,  
\emph{J. London Math. Soc.} {\bf 90} (2014), 333--349.


\bibitem{gj} J. E. Graver, W. B. Jurkat, The module structure of integral designs, {\em J. Combin. Theory} (A) {\bf 15} (1973), 75--90.

\bibitem{Hamada}
    N. Hamada,
    The rank of the incidence matrix of points and $d$-flats in finite geometries,
    \emph{J. Sci. Hiroshima Univ.} {\bf 32} (1968), 381--396.


\bibitem{n11} 
G. D. James, 
{\em Representations of general linear groups}, 
London Mathematical Society, Lecture Note Series 94, 1984.

\bibitem{n8} 
W. M. Kantor,
On incidence matrices of finite projective and affine spaces,
\emph{Math. Z.}  {\bf 124} (1972), 315--318.

  \bibitem{Katona}  
  G. Katona, A theorem of finite sets,
  {\em Theory of graphs (Proc. Colloq., Tihany, 1966)}, (1968), 187--207.  

\bibitem{n6}  
P. Keevash, Shadows and intersections: stability and new proofs,
\emph{Adv. Math.} {\bf 218} (2008),1695--1703.

\bibitem{Keevash} 
P. Keevash, Addendum to Shadows and intersections: stability and new proofs,  January 2010, Available at http://www.maths.qmul.ac.uk/~keevash/papers/kk-addendum.pdf.

\bibitem{Kruskal}
 J. B. Kruskal, The number of simplices in a complex,
 {\em Mathematical Optimization Techniques} (1963), 251--278.

\bibitem{n5} 
N. Linial and B. L. Rothschild, 
Incidence matrices of subsets - A rank formula,
\emph{SIAM J. Algebraic Discrete Math.} {\bf 2} (1981), 333--340.

  \bibitem{Lo2}  
   L. Lov\'{a}sz,
   {\em Combinatorial Problems and Exercises.}
   North-Holland, Amsterdam, 1993. 
   
\bibitem{Serre}
    J. P. Serre,
    \emph{Linear Representations of Finite Groups.}
    Graduate Texts in Mathematics, Springer, 1977.
    

 \bibitem{Wilson1}    
    R. M. Wilson, 
    The exact bound in the Erd\H{o}s-Ko-Rado theorem,
    \emph{Combinatorica}  {\bf 4} (1984), 247--257.

\bibitem{n2} 
R. M. Wilson,
A diagonal form for the incidence matrices of $t$-subsets v. $k$-subsets,
\emph{Europ. J. Combin.} {\bf 11} (1990), 609--615.

\end{thebibliography}
\end{document}